\documentclass[a4paper,reqno, 11pt]{amsart}

\usepackage[margin=3cm]{geometry}

\makeatletter
\let\old@setaddresses\@setaddresses
\def\@setaddresses{\bigskip\bgroup\parindent 0pt\let\scshape\relax\old@setaddresses\egroup}
\makeatother

\usepackage{amssymb} 
\usepackage{amsmath} 
\usepackage{amsthm} 
\usepackage[unicode=true]{hyperref}
\hypersetup{
    colorlinks,
    linkcolor={red!50!black},
    citecolor={blue!50!black},
    urlcolor={blue!80!black}
}

\usepackage{todonotes}
\usepackage{bbm}
\usepackage{paralist}
\usepackage[longnamesfirst,numbers,sort&compress]{natbib}
\usepackage[capitalise]{cleveref}
\usepackage[noend]{algorithmic}

\newtheorem{theorem}{Theorem}[section]

\newtheorem{proposition}[theorem]{Proposition}
\newtheorem{lemma}[theorem]{Lemma}

\newtheorem{question}[theorem]{Question}

\usepackage{rsfso}

\newcommand{\eps}{\varepsilon}

\DeclareMathOperator{\N}{\mathbb{N}}

\let\le\leqslant
\let\ge\geqslant
\let\leq\leqslant
\let\geq\geqslant

\begin{document}

\title{Isometric universal graphs}

\author[L.~Esperet]{Louis Esperet}
\address[L.~Esperet]{Laboratoire G-SCOP (CNRS, Univ.\ Grenoble Alpes), Grenoble, France}
\email{louis.esperet@grenoble-inp.fr}

\author[C.~Gavoille]{Cyril Gavoille}
\address[C.~Gavoille]{LaBRI, University of Bordeaux, Bordeaux, France}
\email{gavoille@labri.fr}

\author[C. Groenland]{Carla Groenland}
\address[C. Groenland]{Utrecht University, Utrecht, The Netherlands}
\email{c.e.groenland@uu.nl}

\thanks{L.\ Esperet is partially supported by the French ANR Projects GATO (ANR-16-CE40-0009-01), GrR (ANR-18-CE40-0032), and by LabEx PERSYVAL-lab (ANR-11-LABX-0025). C.\ Gavoille was partially supported by the French ANR projects DESCARTES (ANR-16-CE40-0023) and DISTANCIA (ANR-17-CE40-0015).
C.\ Groenland is partially supported by the project CRACKNP that has received funding from the European Research Council (ERC) under the European Union’s Horizon 2020 research and innovation programme (grant agreement No 853234). }

\date{\today}
\sloppy

\begin{abstract}
  A subgraph $H$ of a graph $G$ is isometric if the
  distances between vertices in $H$ coincide with the distances
  between the corresponding vertices in $G$.
We show that for any integer $n\ge 1$, there is a graph on
$3^{n+O(\log^2 n)}$ vertices that contains  all $n$-vertex
graphs as isometric subgraphs. Our main tool is a new type of distance labelling scheme, whose study might be of independent interest.
\end{abstract}

\maketitle

\section{Introduction}\label{sec:intro}

\subsection{Universal graphs}

A graph $H$ is said to be \emph{induced-universal} for a graph class
$\mathcal{C}$ if $H$ contains all graphs $G\in \mathcal{G}$ as induced
subgraphs. Let $\mathcal{G}$ denote the class of all graphs, and
$\mathcal{G}_n$ denote the class of all $n$-vertex graphs. It was proved by Moon~\cite{moon65} in 1965 that $\mathcal{G}_n$ has an induced-universal graph with $n\cdot 2^{n/2}$
vertices, and that any induced-universal graph for $\mathcal{G}_n$ must contain at least $2^{(n-1)/2}$
vertices. After intermediate results by Alstrup, Kaplan, Thorup and Zwick~\cite{alstrup.kaplan.ea:adjacency}, Alon~\cite{Alon17} recently
proved that $\mathcal{G}_n$  has an induced-universal graph with
$(1+o(1))2^{(n-1)/2}$ vertices, showing that the lower bound of Moon (which follows
from a simple counting argument) can be attained, up to a lower order
term.

A stronger notion of induced-universal graph is the following: we say that $H$
is an \emph{isometric-universal} graph for a class $\mathcal{C}$ if
$H$ contains isometric copies of all graphs $G\in \mathcal{C}$, where
a subgraph $G$ of $H$ is \emph{isometric} if the distances between
vertices of $G$ are the same in $G$ and $H$: for any $u,v$ in $V(G)\subseteq
V(H)$, $d_G(u,v)=d_H(u,v)$ (where $d_G(u,v)$ denotes the distance
between $u$ and $v$ in $G$). Note that an isometric copy of a graph
$G$ in a graph $H$ is an \emph{induced copy} of $G$ in $H$, as two
vertices are adjacent in a graph if and only if they are at distance
1 in this graph. This implies that any isometric-universal graph for
a class $\mathcal{C}$ is also induced-universal for $\mathcal{C}$. It
turns out that the property of being isometric-universal is
significantly stronger than the property of being
induced-universal. For instance,
Bollob\'as and Thomason~\cite{BT81} proved that the random graph
$G(N,\tfrac12)$ with $N=n^2\cdot 2^{n/2}$ is almost surely induced-universal
for $\mathcal{G}_n$, but since it has diameter 2 almost surely, $G(N,\tfrac12)$  only contains graphs of diameter at most 2 as isometric subgraphs.


\medskip

The
following natural question was recently raised by Peter Winkler
(personal communication).

\begin{question}\label{q:winkler}
Is there a constant $c>1$ such  that the class $\mathcal{G}_n$ of all $n$-vertex graphs has an isometric-universal graph on at most $c^{n}$ vertices?
\end{question}

The main result of the present note is a positive answer to
\cref{q:winkler}, for any $c>3$.


\begin{theorem}\label{thm:main2}
For any integer $n\ge 0$, the class $\mathcal{G}_n$ of all $n$-vertex graphs has an isometric-universal graph on at most $3^{n+O(\log^2 n)}$ vertices.
\end{theorem}

We prove \cref{thm:main2} by studying a new type of
labelling scheme, as we explain next.

\subsection{Labelling schemes}
For a set $S$, and an integer $k\ge 0$, the $k$-fold Cartesian product
$S\times \dots \times S$ is denoted by $S^k$. We write $S^{\le k}=\bigcup_{i=0}^k S^i$ and 
$S^*=\bigcup_{i=0}^\infty S^i$ (i.e.\ $S^*$ denotes the set of finite sequences of elements of $S$, or equivalently, the set of finite words, or strings, on the alphabet $S$). For instance
$\{0,1\}^*$
 denotes the set of finite binary strings, while $(\mathbb{N}\cup\{\infty\})^*$ denotes the set
 of finite sequences whose elements are integers or $\infty$. For a string $s\in S^*$, the length of $s$ is denoted by $|s|$. Throughout the paper, $\log{n}$ denotes the binary logarithm of $n$.  
 
An \emph{adjacency labelling
  scheme} for a graph class $\mathcal{C}$ is a function
$A:\{0,1\}^*\times \{0,1 \}^*\to \{0,1 \}$ such that for any graph $G\in
\mathcal{C}$ there is a function $\ell_G:V(G)\to
\{0,1 \}^*$ with the following property: for any pair of vertices
$u,v\in V(G)$, $A(\ell_G(u),\ell_G(v))=1$ if and only if $u$ and $v$ are
adjacent in $G$. In other words, we can tell whether $u$ and $v$ are
adjacent in $G$ by only looking at the labels $\ell_G(u)$ and $\ell_G(v)$.
Note that the function $A$ depends on
$\mathcal{C}$ (and not on a specific graph $G\in \mathcal{C}$). We say that the adjacency labelling scheme has \emph{labels of at most $k$ bits} if $|\ell_G(v)|\le k$ for
any $G\in \mathcal{C}$ and $v\in V(G)$.

\medskip


Kannan, Naor, and
Rudich~\cite{kannan.naor.ea:implicit0,kannan.naor.ea:implicit} noticed
the following connection between adjacency labelling schemes and
induced-universal graphs.

\begin{theorem}[\cite{kannan.naor.ea:implicit0,kannan.naor.ea:implicit}]\label{thm:knr}
  For any integer $k\ge 0$, a class $\mathcal{C}$ has an adjacency labelling scheme with labels
  of at most $k$ bits if and only if  $\mathcal{C}$ has an induced-universal
  graph with at most $2^{k+1}-1$ vertices.
\end{theorem}

The equivalence is proved as follows. Given an
adjacency labelling scheme $A$ with labels
  of at most $k$ bits, we define an induced-universal graph $H$
  with vertex-set $\{0,1\}^{\le k}$ by connecting any pair of vertices
  $a,b$ by an edge in $H$ if and only if $A(a,b)=1$. For any graph
  $G\in \mathcal{C}$, the
  labelling function $\ell_G:V(G)\to \{0,1\}^{\le k}$ gives a natural
  embedding of $G$ into $H$, and it easily follows from the definition
  of $A$ that the image of $G$ by $\ell_G$ in $H$ is an induced copy of
  $G$ in $H$. Conversely, given an induced-universal graph $H$ for 
  $\mathcal{C}$ with $|V(H)|\le 2^{k+1}-1$
  vertices, we can identify $V(H)$ with (a subset of) $\{0,1\}^{\le k}$, and
  define $A(a,b)=1$ if and only if $a$ and $b$ exist and are adjacent
  in $H$. For any  graph $G\in\mathcal{C}$, any embedding
  of $G$ as an induced copy in $H$ naturally defines a labelling
  $\ell_G:V(G)\to V(H)\subseteq\{0,1\}^{\le k}$ such that for any $u,v\in V(G)$,
  $A(\ell_G(u),\ell_G(v))=1$ if and only if $u$ and $v$ are adjacent in
  $G$.
  
  \smallskip

  Adjacency labelling schemes have been the main tool to construct
  induced-universal graphs with few vertices~\cite{Alon17,
    alstrup.kaplan.ea:adjacency,bonamy.gavoille.ea:shorter,AdjacencyLabellingPlanarFOCS,gavoille.labourel:shorter,kannan.naor.ea:implicit0,kannan.naor.ea:implicit}. As
  a consequence,
  a natural attempt to answer \cref{q:winkler}
  would be to find a type of labelling
  scheme that would be equivalent to isometric-universal graphs. A
  natural candidate is the notion of distance labelling scheme,
  introduced by Gavoille, Peleg, Pérennes and Raz in~\cite{GPPR04} (inspired by the work of Graham and Pollak~\cite{GP72} in 1972, see also~\cite{Winkler83}),
  and further studied
  in~\cite{alstrup.dahlgaard.ea:sublinear,alstrup.gortz.ea:distance,alstrup.gavoille.ea:simpler,alstrup.bille.ea:labeling,GKU16,GU16}. A \emph{distance labelling
  scheme} for a graph class $\mathcal{C}$ is a function
$B:\{0,1\}^*\times \{0,1 \}^*\to \mathbb{N}\cup\{\infty\}$ such that for any graph $G\in
\mathcal{C}$ there is a labelling
function $\ell_G:V(G)\to
\{0,1 \}^*$ with the following property: for any pair of vertices
$u,v\in V(G)$, $B(\ell_G(u),\ell_G(v))=d_G(u,v)$. In other words, we can
determine the distance between $u$ and $v$ in $G$  using only the labels $\ell_G(u)$ and $\ell_G(v)$. As before, if
there is an integer $k\ge 0$ such that $|\ell_G(v)|\le k$ for
any graph $G\in \mathcal{C}$ and $v\in V(G)$, then we say that $\mathcal{C}$ admits an \emph{distance labelling
  scheme with labels of at most $k$ bits}.

Note that a distance labelling scheme tells us in particular whether two
vertices are at distance 1 (equivalently, if they are adjacent), and
thus a distance labelling scheme is also an adjacency labelling scheme.
On the other hand, we have the following partial analogue of
\cref{thm:knr}.

\begin{lemma}\label{lem:distknr}
If a class $\mathcal{C}$ has  an isometric-universal
  graph with at most $2^{k+1}-1$ vertices, for some integer $k\ge 0$, then $\mathcal{C}$ has a distance labelling scheme with labels
  of at most $k$ bits.
\end{lemma}

\begin{proof}
As above, given an isometric-universal
  graph $H$  with at most $2^{k+1}-1$ vertices for $\mathcal{C}$, we define a distance labelling
  scheme $B$ for $\mathcal{C}$ as follows. We identify the vertex set
  of $H$ with (a subset of) $\{0,1\}^{\le k}$, and for any graph $G\in
  \mathcal{C}$ we consider an isometric embedding $\ell_G:V(G)\to V(H)\subseteq \{0,1\}^{\le k}$ of $G$ in $H$. Given $a,b\in V(H)\subseteq \{0,1\}^{\le k}$, we
  simply define $B(a,b)=d_H(a,b)$. It follows from the definition of
  an isometric-universal graph that the distance between $u$ and $v$
  in a graph $G\in \mathcal{C}$ coincides with the distance between
  $\ell_G(u)$ and $\ell_G(v)$ in $H$, so $B$ is indeed a distance
  labelling scheme for $\mathcal{C}$, with labels of at most $k$ bits.
\end{proof}

Interestingly, in this case the connection between labelling schemes
and universal graphs does not go in both directions: distance labelling
schemes cannot be automatically converted into isometric-universal
graphs\footnote{On the other hand, distance labelling schemes can be converted into small universal distance matrices in a natural way, see~\cite{GP03}.}. For instance, the distance labelling scheme of
Winkler~\cite{Winkler83} (see also~\cite{GPPR04}) leads to a graph with constant
diameter, so it can only contain isometric copies of graphs with constant
diameter. 

\medskip

In \cref{sec:distancevector} we define a new type of labelling scheme,
called \emph{distance-vector labelling scheme}, and prove that having such a
scheme with labels of $k$ bits implies the existence of
isometric-universal graphs with $2^k$ vertices. We then show how to
obtain distance-vector labelling schemes with labels of
$O(n)$ bits for all $n$-vertex graphs, which directly implies a positive answer to \cref{q:winkler}. We also explore the limitations of this approach. In \cref{sec:proof2} we prove \cref{thm:main2}. The proof does not use distance-vector labelling schemes but a slightly more technical variant. The generality of the proof also allows us to deduce improved bounds on the size of isometric-universal graphs for families with sublinear separators, such as planar graphs or more generally graphs avoiding some fixed minor. We conclude with some open problems in \cref{sec:ccl}. 

\section{Distance-vector labelling schemes}\label{sec:distancevector}

A \emph{distance-vector labelling
  scheme} for a graph class $\mathcal{C}$ is a function
$D:\{0,1\}^*\to (\mathbb{N}\cup\{\infty\})^*$ such that for any graph $G\in
\mathcal{C}$ there is an ordering $v_1,v_2,\ldots,v_n$ of the vertices of
$G$ and a function $\ell_G:V(G)\to
\{0,1 \}^*$ with the following property: for any vertex $v\in V(G)$,
$D(\ell_G(v))=(d_G(v,v_1), d_G(v,v_2),\ldots, d_G(v,v_n))$. In other words, we can
determine the distance in $G$ between $v$ and each vertex of $G$  by only looking at
the label $\ell_G(v)$. As before, if
there is an integer $k\ge 0$ such that
$|\ell_G(v)|\le k$ for
any graph $G\in \mathcal{C}$ and $v\in V(G)$, then we say that $\mathcal{C}$ admits a \emph{distance-vector labelling
  scheme with labels of at most $k$ bits}.

\smallskip

We  note that contrary to adjacency labelling schemes and distance
labelling schemes, in distance-vector labelling schemes the function
$D$ has a single parameter. 
  

  \medskip

  We start by observing that any distance-vector labelling
  scheme can be translated into a distance labelling scheme with labels
  of the same size. 

  \begin{proposition}\label{obs:dvd}
Let $\mathcal{C}$ be a graph class with a distance-vector labelling
scheme with labels of at most $k$ bits, for some integer $k\ge 0$. Then $\mathcal{C}$ has a distance labelling
scheme with labels of at most $k$ bits.
   \end{proposition}

  \begin{proof}
Let $D$ be a distance-vector labelling
scheme for $\mathcal{C}$ with labels of at most $k$ bits. Consider
a graph $G\in \mathcal{C}$ and let $v_1,\ldots,v_n$ be the associated
ordering of the vertices of $G$, and let $\ell_G: V(G)\to
\{0,1 \}^*$ be the associated labelling function. We now define a
distance labelling scheme $B$ for $\mathcal{C}$. We keep the same
labelling functions $(\ell_G)_{G\in \mathcal{C}}$. For two vertices
$u,v \in G$, we start by considering
$D(\ell_G(u))=(d_G(u,v_1),\ldots,d_G(u,v_n))$ and
$D(\ell_G(v))=(d_G(v,v_1),\ldots,d_G(v,v_n))$. In the first sequence,
the unique index $i$ such that $d_G(u,v_i)=0$ is such that $u=v_i$, so we can
find $d_G(v,v_i)=d_G(v,u)$ in the second sequence. This shows how
to obtain $d_G(u,v)$ from $\ell_G(u)$ and $\ell_G(v)$. So the
implicitly defined function $B(\ell_G(u),\ell_G(v))$ is indeed a
distance labelling scheme for $\mathcal{C}$ with labels of at most $k$ bits, as desired.
  \end{proof}
  
For every vector $x=(x_i)_{i=1}^n\in (\mathbb{N}\cup\{\infty\})^n$, let $\|x\|_\infty=\max_{i=1}^n|x_i|\in \mathbb{N}\cup\{\infty\}$. Adopting the convention that $\infty-\infty=0$, we observe that $(x,y)\mapsto \|x-y\|_\infty$ defines a distance in $(\mathbb{N}\cup\{\infty\})^n$. We now prove that the existence of distance-vector labelling schemes
with small labels
implies the existence of small isometric-universal graphs. 

\begin{lemma}\label{lem:dvlsug}
If a graph class $\mathcal{C}$ has a distance-vector labelling
scheme with labels of at most $k$ bits, for some integer $k\ge 0$, then   $\mathcal{C}$ has an
isometric-universal graph with at most $2^{k+1}-1$ vertices.
\end{lemma}

\begin{proof}
Let $D$ be a distance-vector labelling
scheme for $\mathcal{C}$ with labels of at most $k$ bits. 
Let $H$
be the graph with vertex set $\{0,1\}^{\le k}$, where two vertices
$a,b\in \{0,1\}^{\le k}$ are adjacent in $H$ if and only
if $D(a)$ and $D(b)$ have the same length and $\|D(a)-D(b)\|_\infty= 1$.


Consider two vertices
$a,b\in V(H)=\{0,1\}^{\le k}$ lying in the same connected component of
$H$, and let $a_0,a_1,\ldots,a_t$ be a shortest path between $a=a_0$
and $b=a_t$ in $H$. Note that all the vectors $D(a_i)$, for $0\le i
\le t$, have the same length. Moreover, for any $1\le i \le t$, $\|D(a_{i-1})-D(a_i)\|_\infty= 1$ and thus it follows from the triangle inequality that 
\[
\|D(a)-D(b)\|_\infty\le \sum_{i=1}^t \|D(a_{i-1})-D(a_i)\|_\infty =t =d_H(a,b).
\]

Consider
a graph $G\in \mathcal{C}$ and let $v_1,\ldots,v_n$ be the
associated sequence of 
vertices of $G$, and let  $\ell_G: V(G)\to
\{0,1 \}^{\le k}=V(H)$ be the associated labelling function.

We now prove that $\ell_G$ maps $G$ to an isometric copy of $G$
in $H$.
By definition,
$D(\ell_G(v))=(d_G(v,v_1),\ldots, d_G(v,v_n))$ for any vertex $v\in
V(G)$. 
If $uv\in E(G)$, then $u\neq v$ so
$\|D(\ell_G(u))-D(\ell_G(v))\|_{\infty}\geq 1$. Moreover,
$|d_G(u,v_i)-d_G(v,v_i)|\leq d_G(u,v)=1$ by the triangle inequality
for all $i$, and thus $\|D(\ell_G(u))-D(\ell_G(v))\|_{\infty}=
1$. Hence we find that $G$ embeds as a subgraph of $H$ via
$\ell_G$, and thus $d_H(\ell_G(u),\ell_G(v))\le d_G(u,v)$ for any $u,v\in V(G)$. We now prove that for all $u,v\in V(G)$, any path between
$\ell_G(u)$ and $\ell_G(v)$ in $H$ has length at least $d_G(u,v)$,
which shows that $G$ is an isometric subgraph of $H$.


Let $1\le i,j\le n$ be indices such that $u=v_i$ and $v=v_j$. We know
that the $j$-th entry $D(\ell_G(v))_j$ of the vector $D(\ell_G(v))$ is
equal to $d_G(v,v_j)=d_G(v,v)=0$, while the $j$-th entry $D(\ell_G(u))_j$ of the vector $D(\ell_G(u))$ is equal to $d_G(u,v_j)=d_G(u,v)$ and so
\[
\|D(\ell_G(u))-D(\ell_G(v))\|_\infty\ge |D(\ell_G(u))_j-D(\ell_G(v))_j|= d_G(u,v).
\]
This shows that $d_H(\ell_G(u),\ell_G(v))\ge \|D(\ell_G(u))-D(\ell_G(v))\|_\infty\ge  d_G(u,v)$. 
\end{proof}

We now show how to produce distance-vector labelling schemes with small
labels. It will be convenient to restrict ourselves to connected
graphs, but as the next proposition shows, we will not lose much
generality by doing so.

\begin{proposition}\label{obs:connected}
Assume that for some integer $n\ge
1$, the class of connected graphs with at most $n$
vertices has an
isometric-universal graph $G_n$ with at most $g(n)$ vertices. Then the class $\mathcal{G}_n$ of all $n$-vertex graphs has an
isometric-universal graph $H_n$ with at most $n\cdot g(n)$ vertices. 
  \end{proposition}

To see this, it suffices to define $H_n$ as the disjoint union of $n$
copies of $G_n$. Clearly, each of the (at most $n$) connected components
of any  graph $G\in \mathcal{G}_n$ embeds as an isometric subgraph in a
different copy of $G_n$ in $H_n$, and the resulting embedding is an isometric
embedding of $G$ in $H_n$. 

Note that we could be more precise here: when $g(n)=c^n$ for some $c>0$,  the bound $n \cdot g(n)$ in \cref{obs:connected} can be replaced by $(1+o(1))\cdot g(n)$, by considering isometric-universal graphs for connected graphs of size $n, n/2, n/3, \ldots ,1$ instead (assuming such isometric-universal graphs exist for all these values). However this would not change the lower order terms in our constructions, so we prefer to use the simpler bound $n \cdot g(n)$.

\medskip

We start with a simple distance-vector labelling scheme with labels of at most 
$(4+o(1))n$ bits (leading to an isometric-universal graph of
$(16+o(1))^n$ vertices for the class $\mathcal{G}_n$).
The proof follows the lines of the proof of~\cite[Lemma 2.2]{GKU16} for distance labelling schemes; we include it for the convenience of the reader and since our analysis is slightly simpler due to the fact that we have no requirements on the decoding time. With the additional arguments from \cite{GKU16}, constant decoding time could be achieved if desired. Moreover, we can improve the $4n$ above to $3n$ by adapting the proof of the follow-up paper \cite{alstrup.gavoille.ea:simpler}. 

\begin{theorem}\label{thm:4n}
For any integer $n\ge 1$, the class of all connected  $n$-vertex graphs has a distance-vector labelling
scheme with labels of at most $4n+O(\log n)$ bits.
\end{theorem}


\begin{proof}
Let $G$ be a connected $n$-vertex graph. 
It is well known that there is a tour visiting all vertices of $G$ that uses at most $2n$ edges. Indeed, consider any spanning tree $T$ of $G$, double every edge of $T$ and note that the resulting graph is Eulerian; the corresponding Eulerian walk gives the desired tour. 
In particular, if we order the vertices $v_1,\dots,v_n$ according to their first appearance in the tour (fixing an arbitrary starting vertex $v_1$) then 
\[
d_G(v_1,v_2)+\dots +d_G(v_{n-1},v_n)\leq 2n.
\]
For any vertex $v\in G$, in order to encode the
distances $d_G(v,v_i)$, for all $i=1,\dots,n$, it is sufficient to record $d_G(v,v_1)$, and
for any $2\le i \le n$, 
$\delta_i=d_G(v,v_i)-d_G(v,v_{i-1})$. From the triangle inequality, we find that
\[
\sum_{i=2}^n|d_G(v,v_i)-d_G(v,v_{i-1})|\leq \sum_{i=2}^n d_G(v_i,v_{i-1})\leq 2n.
\]
We use $n-1$ bits to store the signs of $\delta_2,\dots,\delta_n$. For their absolute values, we note that there is a simple bijection between sequences of integers $b_1,\dots,b_{n-1}\geq 0$ satisfying $\sum_{i=1}^{n-1} b_i \leq 2n$ and binary sequences of length at most $3n$ with exactly $n$ 1's (it suffices to write a 1 followed by $b_i$ 0's, for each $i=1,\ldots,n$ in order).
In total, we use at most $n-1+3n+O(\log n)=4n+O(\log n)$ bits, where we used a further $\lceil \log n\rceil$ bits in order to record $d_G(v,v_1)$.
\end{proof}

Note that the bound $4n$ above can easily be optimized in several different ways, but here we chose to present a simplest possible proof instead.
\cref{thm:4n} directly implies the following exponential upper bound on the size of an isometric-universal graph for $\mathcal{G}_n$, providing a positive answer to \cref{q:winkler}.

\begin{theorem}
For any integer $n\ge 1$, the class $\mathcal{G}_n$ of all $n$-vertex graphs has an isometric-universal graph on at most $16^{n+O(\log n)}$ vertices.
\end{theorem}

\begin{proof}
Let $n\ge 1$ be an integer.  \cref{thm:4n} and \cref{lem:dvlsug} imply that the class of
  connected $n$-vertex graphs has an isometric-universal graph
  $G_n$ with at most $2^{4n+O(\log n)}$ vertices. Since any connected graph of at most $n$ vertices is an isometric subgraph
of some connected $n$-vertex graph, $G_n$ is isometric-universal for
the class of connected graph with at most $n$ vertices. By
\cref{obs:connected}, this shows that the class $\mathcal{G}_n$
has an isometric-universal graph with at most $n\cdot 2^{4n+O(\log n)}=16^{n+O(\log n)}$
vertices, as desired.
\end{proof}

A natural problem is to determine the smallest constant $c>0$ such that the class $\mathcal{G}_n$ has a distance-vector labelling scheme with labels of at most $cn$ bits. While simple counting arguments show that adjacency labelling schemes for $\mathcal{G}_n$ require labels of $(n-1)/2$ bits~\cite{moon65}, the unary nature of distance-vector labelling scheme allows us to show that in our case, $c\ge 1$ is the natural lower bound.

\begin{theorem}
Any distance-vector labelling scheme for the class $\mathcal{G}_n$ of all $n$-vertex graphs  needs labels of at least $(1-o(1))n$ bits.
\end{theorem}
\begin{proof}
Let $0<\epsilon<\frac12$ and $n\in \N$. Suppose for convenience that $\epsilon n$ is an integer. Consider the family $\mathcal{B}_n$ of $n$-vertex bipartite graphs with a part of size $\epsilon n$ and another part of size $(1-\epsilon)n$. Since the complete bipartite graph in $\mathcal{B}_n$ contains $\epsilon n\cdot (1-\epsilon)n=\epsilon(1-\epsilon)n^2$ edges,
there are at least 
\[
\frac{2^{\epsilon(1-\epsilon)n^2}}{n!}\geq 2^{\epsilon (1-\epsilon)n^2-n\log n}
\]
isomorphism types in $\mathcal{B}_n$. 

Suppose we have a distance-vector labelling with labels of at most $f(n)$ bits for $\mathcal{G}_n$. Since $\mathcal{B}_n\subset \mathcal{G}_n$, we can use this scheme to encode the graphs in $\mathcal{B}_n$ as follows. Given $G\in \mathcal{B}_n$, there is an ordering $v_1,\dots,v_n$ of the vertices such that each vertex has a label of at most $f(n)$ bits from which we can decode the distances to all the vertices. Consider the binary string obtained by concatenating the labels of the vertices of the partite set of size $\epsilon n$. This binary string has size at most $\eps n\cdot f(n)$, and it can be observed that it is enough to reconstruct (an isomorphic copy of) $G$. Indeed, the labels telling the distances tell in particular the index of each vertex (the unique vertex at distance 0) and the neighbors of each vertex (the set of vertices at distance 1).
This shows that
\[
\epsilon n \cdot f(n) \geq \epsilon(1-\epsilon)n^2-n\log n,
\]
which implies that $f(n)\geq (1-\epsilon)n-\epsilon^{-1}\log n$.
\end{proof}

Together with \cref{thm:4n}, this shows that the smallest constant $c$ such that the class $\mathcal{G}_n$ has a distance-vector labelling scheme with labels of at most $c\cdot n$ bits satisfies $1\le c \le 4$ (again, we can decrease the bound 4 in  \cref{thm:4n} at the cost of a more careful analysis, but currently not beyond 3). As our main result will be proved using a different type of distance-vector labelling schemes, we do not try to obtain the best constant $c$ here (although the problem of optimizing $c$  might be interesting in its own right, see~\cref{sec:ccl}).

\medskip

In the remainder of this paper, we prove \cref{thm:main2}. As alluded to above, instead of using distance-vector labelling schemes directly, we consider a technical variant in which each vertex only records its distance to a certain subset of ancestors. On the way, we observe that distance labelling schemes constructed in~\cite{GPPR04,GU16} for graph classes with sublinear separators can be adapted to construct small isometric-universal graphs for these classes.



\section{Proof of \texorpdfstring{\cref{thm:main2}}{Theorem 1.2}}\label{sec:proof2}

Given a  graph $G$, assume that there is  a rooted tree $T$ and a partition $(B_t)_{t\in V(T)}$ of the vertex set of $G$ into non-empty sets  (called \emph{bags}) indexed by the nodes of $T$. Recall that the ancestors of a node $t\in T$ are the nodes lying on the unique path from the root of $T$ to $t$ in $T$ (we consider $t$ to be an ancestor of itself). Given a vertex $v\in V(G)$, and a pair $(T,(B_t)_{t\in V(T)})$ as above, let $t\in V(T)$ be such that $v\in B_t$. Then $B_t$ is called the \emph{bag of $v$} and all the bags $B_{t'}$ such that $t'$ is an ancestor of $t$ in $T$ are called the \emph{ancestor bags} of $v$ and $B_t$. 

\smallskip

A pair $(T,(B_t)_{t\in V(T)})$ as above is called a \emph{hierarchical decomposition} of $G$
if for each edge $uv\in E(G)$, $u$ lies in an ancestor bag of $v$, or vice-versa.

\smallskip

Given an ordering $v_1,v_2,\ldots,v_n$ of the vertices of a graph $G$, the \emph{$V(G)$-index} of a vertex $v\in V(G)$ is the integer $1\le j \le n$ such that $v=v_j$. Assume we have an ordering $v_1,v_2,\ldots,v_n$ of the vertices of a graph $G$, and a hierarchical decomposition $(T,(B_t)_{t\in V(T)})$ of $G$. Let $v\in V(G)$. We say that a vertex $u\in V(G)$ is an \emph{ancestor} of $v$ (with respect to the decomposition $(T,(B_t)_{t\in V(T)})$ and the ordering $v_1,\ldots,v_n$), if $u$ lies in a strict ancestor bag of $v$ (i.e.\ in an ancestor bag of $v$ distinct from the bag of $v$), or if $u$ and $v$ lie in the same bag and the $V(G)$-index of $u$ is at most the $V(G)$-index of $v$. If the decomposition and the ordering are clear from the context, we simply say that $u$ is an ancestor of $v$.

Note that for each vertex $v$, the set of ancestors of $v$ is totally ordered by the ancestor relation (as this relation is transitive, and for any two ancestors $u,w$ of $v$, one of $u,w$ is an ancestor of the other). The corresponding ordering of the ancestors of $v$ is called \emph{the natural ordering of the ancestors of v} with respect to the hierarchical decomposition $(T,(B_t)_{t\in V(T)})$ and the ordering $v_1,\ldots,v_n$ (again when the decomposition and the ordering are clear from the context we omit them in the terminology). An equivalent way to consider this ordering is the following: if the ancestor bags of $v$ are $B_{t_1},\ldots,B_{t_k}$ in order, where $t_1$ is the root of $T$ and $B_{t_k}$ is the bag of $v$, then the natural ordering of the ancestors of $v$ corresponds to enumerating, for each $i=1,\ldots,k$ in order, the vertices of $B_{t_i}$, where the vertices in each bag are sorted according to their $V(G)$-indices and for the bag $B_{t_k}$ of $v$ we only consider the vertices of $V(G)$-index at most the $V(G)$-index of $v$. Note that $v$ is always the final vertex in the natural ordering of its ancestors.

\medskip

Let $\mathcal{C}$ be a class of graphs. 
Assume that there is a decoding function $D:\{0,1\}^*\to
\mathbb{N}^*$ such that the following holds. For each $G\in \mathcal{C}$, there is an ordering $v_1,v_2,\ldots,v_n$ of the vertices of $G$, a labelling function $\ell_G:V(G)\to \{0,1\}^*$, and a hierarchical decomposition $(T_G,(B_t)_{t\in V(T_G)})$ of $G$, such that for any $v\in V(G)$, $D(\ell_G(v))=(p(v),x(v))$, where 
\begin{itemize}
    \item $p(v)\in \{1,\ldots,n\}^{\le n}$ is a vector such that for any $1\le i\le |p(v)|$, the $i$-th entry of $p(v)$ (denoted by $p(v)_i$) is the $V(G)$-index of the $i$-th vertex in the natural ordering of the ancestors of $v$.
    \item $x(v)\in \{0,\ldots,n\}^{\le n}$ is a vector with $|x(v)|=|p(v)|$, such that for any $1\le i \le |x(v)|$, the $i$-th entry of $x(v)$  is equal to $d_G(v,v_j)$, where $j=p(v)_i$.
\end{itemize}

In other words, $D(\ell_G(v))$ allows us to find the indices of the ancestors  of $v$ in the decomposition, from the root of $T_G$ to $v$, and the distances from $v$ to each of these vertices in $G$.  We call this a \emph{hierarchical distance-vector labelling scheme} for $\mathcal{C}$. As before, if $|\ell_G(v)|\le k$ for all $G\in\mathcal{C}$  and $v\in V(G)$, then we say that the scheme has labels of at most $k$ bits. 

\medskip

Hierarchical distance-vector labelling schemes (and distance-vector labelling schemes) can be seen as a special case of \emph{hub-labelling}, where each vertex $v$ stores its distances to some set $S_v$ of vertices, in such a way that for any two vertices $u$ and $v$, some shortest path between $u$ and $v$ intersects $S_u\cap S_v$ (see~\cite{KUV19} and the references therein).

\medskip

In the proof of the next result it will be convenient to consider distances between vectors of different lengths. We define the \emph{$L^\infty$-pseudodistance} between two vectors $x,y\in \mathbb{N}^*$ of different lengths as the $L^\infty$-distance between the prefixes  of $x$ and $y$ of length $\min(|x|,|y|)$.

\begin{lemma}\label{lem:td2}
Let $\mathcal{C}$ be a class of graphs with a hierarchical distance-vector labelling scheme with labels of at most $k$ bits, for some integer $k\ge 0$. Then  $\mathcal{C}$ has an isometric-universal graph with at most $2^{k+1}-1$ vertices.
\end{lemma}

\begin{proof}
Let $D$ denote a hierarchical distance-vector labelling scheme for $\mathcal{C}$, with labels of at most $k$ bits, and let $(\ell_G)_{G\in \mathcal{C}}$ denote the associated labelling functions and $(T_G,(B_t)_{t\in V(T_G)})_{G\in \mathcal{C}}$ the associated hierarchical decompositions.
 
We define a graph $H$ whose vertex set consists of all $z\in \{0,1\}^{\le k}$,  such that $D(z)=(p,x)$ exists, and the final entry of $x$ is a 0.



\medskip

The number of vertices in $H$ is at most $2^{k+1}-1$. We define adjacency in $H$ as follows: let $z_1,z_2$ be two vertices of $H$ and let us denote $D(z_1)=(p_1,x_1)$ and $D(z_2)=(p_2,x_2)$. Then $z_1$ is adjacent to $z_2$ in $H$ if and only if 
\begin{itemize}
    \item one of $p_1,p_2$ is a prefix of the other, and
    \item $x_1$ is at $L^\infty$-pseudodistance 1 from $x_2$.
\end{itemize}
We now prove that $H$ is isometric-universal for $\mathcal{C}$. Consider some graph $G\in \mathcal{C}$, and let $v_1,\ldots,v_n$ be the ordering of the vertices of $G$ associated to the decoding function $D$.  We write $T=T_G$ for the rooted tree in the hierarchical decomposition of $G$ associated to $D$, and $\ell=\ell_G$ for the labelling function. Given a vertex $v\in V(G)$, we map $v$ to $\ell(v)$ in $H$. Note that $|\ell(v)|\in \{0,1\}^{\le k}$ and that $D(\ell(v))$ is defined. Moreover, if we write $D(\ell(v))=(p_v,x_v)$, then since $v$ is the final vertex in the natural ordering of its ancestors, the final entry of $x_v$ is equal to $d_G(v,v)=0$. This shows that $\ell(v)$ is indeed a vertex of $H$.
It remains to prove that this gives an isometric embedding of $G$ in $H$. 

Let $u,v\in V(G)$. We write $D(\ell(u))=(p_u,x_u)$ and
$D(\ell(u))=(p_v,x_v)$. If $uv\in E(G)$, then $u\ne v$ and we may
assume that $u$ is an ancestor of $v$ (since $(T,(B_t)_{t\in V(T)})$
is a hierarchical decomposition of $G$, one of $u,v$ is an ancestor of the other). This implies that $p_u$ is a prefix of $p_v$. Note that $x_u$ is a vector recording the distance from $u$ to each ancestor of $u$, and the prefix of $x_v$ of size $|p_u|=|x_u|$ records the distance between $v$ and the same vertices, in the same order. Since $uv\in E(G)$, it follows from the triangle inequality that for each vertex $w$ in the sequence, $|d_G(u,w)-d_G(v,w)|\le 1$, and thus the two vectors $x_u$ and $x_v$ are at $L^\infty$-pseudodistance at most $1$. Moreover $d_G(u,u)=0$ while $d_G(v,u)=1$, so the two vectors are at $L^\infty$-pseudodistance exactly $1$. 

This shows that $G$ embeds as a subgraph of $H$ via the mapping $u\mapsto \ell(u)$, and thus
\[
d_H\big(\ell(u),\ell(v)\big)\le d_G(u,v)
\]
for any $u,v\in V(G)$. In the remainder of the proof, we show that for all $u,v\in V(G)$, any path between $\ell(u)$ and $\ell(v)$ in $H$ has length at least $d_G(u,v)$,
which implies that $G$ is an isometric subgraph of $H$.

\medskip

First consider a shortest path $z_0,z_1\ldots,z_t$ in $H$, and write
$D(z_i)=(p_i,x_i)$ for any $0\le i \le t$. We first consider the
special situation in which for each $i\ge 0$, $p_0$ is a prefix of
$p_i$. For any $0\le i \le t$, we write $x_i'$ for the prefix of $x_i$
of length $|x_0|$. Note that for any $1\le i \le t$,
$\|x_{i-1}'-x_i'\|_\infty\le 1$ by the definition of $H$, and thus it follows from the triangle inequality that 
\[
\|x_0-x_t'\|_\infty\le \sum_{i=1}^t \|x_{i-1}'-x_i'\|_\infty \le t 
\]
We now consider a shortest path $P=z_0,z_1,\ldots,z_t$ in $H$ between $z_0=\ell(u)$ and $z_t=\ell(v)$ for vertices $u$ and $v$ in some graph $G$. We again write $D(z_i)=(p_i,x_i)$ for any $0\le i \le t$. Let $j\in \{0,\ldots,t \}$ be such that $|p_j|=|x_j|$ is minimal. Since for any $1\le i \le t$, one of $p_i,p_{i-1}$ 
is a prefix of the other, it follows that $p_j$ is a common prefix of all $p_i$, for $0\le i \le t$. For any $0\le i \le t$, we write $x_i'$ for the prefix of $x_i$ of length $|x_j|$. 
By the paragraph above, we obtain that 
\[
\|x_0'-x_j\|_\infty\le  j \text{ and } \|x_j-x_t'\|_\infty\le  t-j, 
\]
Let $w$ be the $|p_j|$-th ancestor of $u$ (in the natural ordering of the ancestors of $u$). By transitivity, since prefixes of $p_0,p_1,\ldots,p_t$ of size $|p_j|$ coincide along the edges of the path, $w$ is also the $|p_j|$-th ancestor of $v$. It follows that the $|p_j|$-th entry in $x_0$ (and $x_0'$) is equal to $d_G(u,w)$, and the $|p_j|$-th entry in $x_t$ (and $x_t'$) is equal to $d_G(v,w)$. By definition of the vertex set of $H$, since $z_j\in V(H)$ and $D(z_j)=(p_j,x_j)$, it follows that the $|p_j|$-th entry of $x_j$ is equal to $0$. This implies that $\|x_0'-x_j\|_\infty \geq |d_G(u,w)-0|=d_G(u,w)$ and similarly $\|x_j-x_t'\|_\infty \geq d_G(w,v)$. As a consequence, 
\begin{align*}
d_G(u,v)\le d_G(u,w)+d_G(v,w) & \le \|x_0'-x_j\|_\infty+\|x_j-x_t'\|_\infty\\
& \le j+t-j=t.
\end{align*}
This shows that $t=d_H\big(\ell(u),\ell(v)\big)\ge d_G(u,v)$, as desired.
 \end{proof}
 
 We now explain how to obtain a hierarchical distance-vector labelling scheme for $\mathcal{G}_n$ with labels of size roughly $\log 3\cdot n$. We will need the following lemma, proved in \cite[Section 4]{alstrup.gavoille.ea:simpler} using classical tools from \cite{ST83}, and which is the main technical ingredient for the construction of a distance labelling scheme with labels of at most $(\tfrac12 \log 3+o(1)) n$ bits in \cite{alstrup.gavoille.ea:simpler}.

 \begin{lemma}[\cite{alstrup.gavoille.ea:simpler}]\label{lem:heavy}
 For any rooted tree $T$, there is a (non necessarily proper) 2-coloring of the vertices of $T$ with colors red and blue, and an ordering $v_1,\ldots,v_n$ of the vertices of $T$ such that the following holds
 \begin{enumerate}
 \item $v_1$ is the root of $T$, and is colored blue
     \item each vertex has $O(\log n)$ blue  ancestors.
     \item for every red vertex $u$, the parent $v$ of $u$ appears directly before $u$ in the ordering: there is an integer $1\le i \le n-1$ such that $v=v_i$ and $u=v_{i+1}$.
 \end{enumerate}
 \end{lemma}
 
Note that a consequence of \cref{lem:heavy} is that for any vertex $v$
in $T$, the path from the root to $v$ can be divided into $O(\log n)$
subpaths, each containing at most one blue vertex, and such that any
two adjacent vertices in any of these subpaths are consecutive in the ordering. 

\begin{theorem}\label{thm:log3n}
For any $n\ge 1$, the class of all connected $n$-vertex graphs  has a hierarchical distance-vector labelling scheme with labels of at most $n \cdot \log 3 +O(\log^2 n)$ bits.
\end{theorem}

\begin{proof}
Let $n\ge 1$ and let $G$ be an $n$-vertex connected graph. Let $T$ be a Depth-First-Search spanning tree of $G$, with root $r$. It is well known that any edge $uv$ in $G$ connects a vertex to one of its ancestors in $T$. So if we define $B_v=\{v\}$ for any vertex $v\in V(G)$, then we obtain that $(T,(B_t)_{t\in V(T)})$ is a hierarchical decomposition of $G$. 

Apply \cref{lem:heavy} to $T$, and let $v_1,\ldots,v_n$ be the
corresponding ordering of the vertices of $T$ (and thus $G$). Let $v$
a vertex of $G$, and let $P=t_1,\ldots,t_k$ be the unique path from
$t_1=r$ to $t_k=v$ in $T$.  Note that the vertices $t_1,\ldots,t_k$
are the ancestors of $v$ not only in $T$, but also in $G$ (with
respect to the hierarchical decomposition $(T,(B_t)_{t\in V(T)})$ and
the ordering $v_1,\ldots,v_n$), and the natural ordering of these
ancestors of $v$ is precisely $t_1,\ldots,t_k$. Let $p(v)\in
\{1,\ldots,n\}^k$ be the vector in which for any $1\le i \le k$, the
$i$-th entry (denoted by $p(v)_i$) is the $V(G)$-index of $t_i$. By
\cref{lem:heavy}, the path $P$ is divided into $O(\log n)$ subpaths in
which all $V(G)$-indices are consecutive. In order to store $p(v)$,
it suffices to store the $V(G)$-indices of the $O(\log n)$ endpoints
of these subpaths. (We allocate a fixed number of bits for this for fixed $n$, so the number of subpaths
does not need to be stored explicitly). It follows that $p(v)$ can be encoded with $O(\log^2 n)$ bits.

Let $x(v)\in \{0,\ldots,n\}^k$ be the vector in which for any $1\le i
\le k$, the $i$-th entry is equal to $d_G(v,t_i)$. To store $x(v)$, we
record the distance $d_G(v,r)=d_G(v,t_1)$ explicitly, using $O(\log
n)$ bits, and for each $2\le i\le k$ we store
$\delta_i=d_G(v,t_i)-d_G(v,t_{i-1})\in \{-1,0,1\}$. As $k\le n$, this
can be recorded in $n \cdot \log 3+O(\log n)$ bits in total. It
follows that the  class of all connected $n$-vertex graphs  has a hierarchical distance-vector labelling scheme with labels of at most $n \cdot \log 3 +O(\log^2 n)$ bits.
\end{proof}

It should be noted that even if we use the same technical tool as the proof of the distance labelling scheme with labels of $(\tfrac12 \log 3+o(1)) n$ bits in \cite{alstrup.gavoille.ea:simpler}, our proof here is quite different. In particular, if $\mathcal{P}_n^k$ is the class of all $n$-vertex graphs with no path of length more than $k$, for some integer $k=k(n)$, then the proof above gives a distance-vector labelling scheme for $\mathcal{P}_n^k$ with labels of at most $k\cdot \log 3 +O(\log^2 n)$ bits, as any Depth-First-Search tree in such a graph has height at most $k$. However, in \cite{alstrup.gavoille.ea:simpler} the bound on the height of the tree does not affect the leading term of the label size of a vertex $v$, which is caused by storing $d_G(v,x)-d_G(v,\text{parent}(x))\in \{-1,0,1\}$ for about $n/2$ vertices $x$.

\medskip

With \cref{thm:log3n} and \cref{lem:td2} in hand, we are now ready to prove \cref{thm:main2}.

\begin{proof}[Proof of \cref{thm:main2}]
Let $n\ge 1$ be an integer.  \cref{thm:log3n} and \cref{lem:td2} imply that the class of
  connected $n$-vertex graphs has an isometric-universal graph
  $G_n$ with at most $2^{n\log 3 +O(\log^2 n)}$ vertices. Since any connected graph of at most $n$ vertices is an isometric subgraph
of some connected $n$-vertex graph, $G_n$ is isometric-universal for
the class of connected graph with at most $n$ vertices. By
\cref{obs:connected}, this shows that the class $\mathcal{G}_n$
has an isometric-universal graph with at most $n\cdot 2^{n\log 3 +O(\log^2 n)}=3^{n+O(\log^2 n)}$
vertices, as desired.
\end{proof}

The generality of hierarchical distance-vector labelling schemes allows us to also derive good bounds on the size of isometric-universal graphs for classes with small separators, as we now explain.

\medskip

A vertex set $S$ in an $n$-vertex graph $G$ is said to be a \emph{balanced separator} if $V(G)-S$ can be partitioned into two sets $X,Y$, each of size at most $2n/3$, such that no edge of $G$ has one endpoint in $X$ and the other in $Y$. It is well known that every tree has a balanced separator consisting of a single vertex (see for instance~\cite{Chu90}), and more generally every graph of bounded treewidth has a balanced separator of constant size. The planar separator theorem of Lipton and Tarjan~\cite{LT79} states that $n$-vertex planar graphs have balanced separators of size $O(\sqrt{n})$, and it was proved that the same holds for any proper minor-closed class~\cite{AST90}.

\medskip

In the remainder of the paper it will be convenient to assume that $\mathcal{C}$ is a \emph{hereditary} class of graphs, that is every induced subgraph of a graph of $\mathcal{C}$ is also in $\mathcal{C}$. Note that
we can decompose any graph $G\in \mathcal{C}$ by constructing some (binary) rooted tree $T_G$ and some partition $(B_t)_{t\in V(T_G)}$ of $V(G)$ inductively as follows. Let $S$ be a non-empty\footnote{Note that any empty balanced separator in an non-empty graph can be converted to a non-empty separator by adding an arbitrary vertex to the separator.} balanced separator of $G$, and let $X,Y$ be a partition of $V(G)-S$ into two sets of at most two thirds of the vertices, with no edges between $X$ and $Y$. Inductively, we construct rooted trees $T_1$ and $T_2$ for $G_1=G[X]$ and $G_2=G[Y]$ respectively, as well as corresponding partitions $(B_t)_{t\in V(T_1)}$ of $X$ and $(B_t)_{t\in V(T_2)}$ of $Y$. We add a root $r$, set $B_r=S$ and then define $T_G$ as the tree with root $r$ having at most two children $t_1$ and $t_2$, so that the subtree rooted in $t_i$ is equal to $T_i$ for $i=1,2$ (note that a vertex $t_i$ does not exist if $T_i$ and the corresponding subgraph of $G$ are empty). It follows from the inductive construction that $(B_t)_{t\in T_G}$ is indeed a partition of $G$. Note that by the definition of separators, for any edge $uv\in E(G)$, $u$ is in some ancestor bag of $v$, or vice-versa. This shows that the pair $(T_G,(B_t)_{t\in V(T_G)})$ constructed in this way is a hierarchical decomposition of $G$. 

\medskip

Given
a class $\mathcal{C}$, we denote by $\mathcal{C}_n$ the class of
$n$-vertex graphs of $\mathcal{C}$. 
We say that a graph class $\mathcal{C}$ has \emph{balanced separators of size at most $f(n)$}, for some nondecreasing function $f:\mathbb{N}\to \mathbb{N}$, if for any $n\ge 1$, any graph $G\in \mathcal{C}_n$ has a balanced separator of size at most $f(n)$. 

\begin{theorem}\label{thm:sep}
Let $\mathcal{C}$ be a hereditary class with balanced separators of size at most $f(n)$. Then for any integer $n\ge 1$, the class $\mathcal{C}_n$ has an isometric-universal graph with at most $2^{O(f(n)\cdot \log^2 n)}$ vertices. 
\end{theorem}

\begin{proof}
Given any graph $G\in \mathcal{C}_n$, let $(T_G,(B_t)_{t\in T_G})$ be a hierarchical  decomposition of $G$ obtained as above, by taking only balanced separators of size at most $f(n)$. Note that by the definition of balanced separators, the height of $T_G$ is  $O(\log n)$. Consider any ordering $v_1,\ldots,v_n$ of the vertices of $G$. For each vertex $v\in V(G)$, we store the $V(G)$-indices of the ancestors of $v$ and the distances from $v$ to these vertices. Note that $v$ has $O(\log n)$ ancestor bags and each contains at most $f(n)$ vertices, so we only need to store $O(f(n)\cdot \log n)$ indices and distances (which are elements of $\{0,\ldots,n\}$, so this takes at most $O(f(n)\cdot \log^2 n)$ bits per vertex).

This gives a hierarchical distance-vector labelling scheme for $\mathcal{C}_n$, with labels of at most $ O(f(n)\cdot \log^2 n)$ bits. By \cref{lem:td2}, this implies that $\mathcal{C}_n$ has an isometric-universal graph with at most $2^{O(f(n)\cdot \log^2 n)}$ vertices, as desired.
\end{proof}

When the separator size $f(n)$ is at least $n^\epsilon$, for some $\epsilon>0$, a multiplicative factor of $\log n$ can be avoided in the exponent by observing that the size of the bags decreases geometrically with the depth in the tree, so each vertex only needs to store distances to $O(f(n))$  ancestors in this case. Using the separator theorem from~\cite{AST90}, this shows that $n$-vertex graphs from any proper minor-closed class have an isometric-universal graph with at most $2^{O(\sqrt{n} \log n)}$ vertices. It is possible to avoid another multiplicative factor of $\log n$ in the exponent in the case of planar graphs, using the ideas of \cite{GU16}, which leads to an isometric-universal graph with at most $2^{O(\sqrt{n})}$ vertices for this class. Using \cref{lem:distknr}, this shows that the best known bounds on distance-labelling schemes for classes with small separators can be obtained from isometric-universal graphs. Since any distance labelling scheme for the class of $n$-vertex planar graphs requires labels of $\Omega(n^{1/3})$ bits~\cite{GPPR04}, \cref{lem:distknr} also shows that any isometric-universal graph for the class of $n$-vertex planar graphs needs $2^{\Omega(n^{1/3})}$ vertices.

\section{Conclusion}\label{sec:ccl}

A natural problem is to find the smallest constant $c$ such that $\mathcal{G}_n$ has an isometric-universal graph on at most $2^{cn}$ vertices. It is possible that $c=\tfrac12+o(1)$, but currently any improvement over the best known constant for distance labelling scheme from~\cite{alstrup.gavoille.ea:simpler}, that is proving that $c<\tfrac12 \log(3)$, would already be significant. As we have mentioned in the introduction, almost all $n$-vertex graphs have diameter 2, so it follows that almost all $n$-vertex graphs embed isometrically in any induced-universal graph for $\mathcal{G}_n$ (with at most $2^{n/2}$ vertices). As a consequence, we only need to consider a vanishing proportion of the graphs in $\mathcal{G}_n$.

\medskip

It was mentioned in the previous section that for the class of $n$-vertex planar graphs, any isometric-universal graph needs at least $2^{\Omega(n^{1/3})}$ vertices. On the other hand, it was proved in \cite{AdjacencyLabellingPlanarFOCS} that the same class has an induced-universal graph with $n^{1+o(1)}$ vertices. So in general the minimum size of an isometric-universal graph for a class $\mathcal{C}$ can be very different from the minimum size of an induced-universal graph for $\mathcal{C}$. However, it might be possible that for dense hereditary classes (classes $\mathcal{C}$ such that $|\mathcal{C}_n|=2^{\Theta(n^2)}$) the two sizes coincide, up to lower order terms (see~\cite{BEGS20} for more on induced-universal graphs for dense hereditary classes). If true, this would in particular imply the existence of isometric-universal graphs for $\mathcal{G}_n$ with $2^{n/2+o(n)}$ vertices, and the existence of a distance labelling scheme for $\mathcal{G}_n$ with labels of at most $n/2+o(n)$ bits. 

\medskip

For distance-vector labelling schemes, which we introduce in this paper, it is possible that labels of $n+o(n)$ bits are sufficient for the class $\mathcal{G}_n$. Proving this would again improve on the best known distance labelling scheme for $\mathcal{G}_n$, but the unary nature of the problem seems to require new tools.

\medskip

Finally, we wonder whether the bound $2^{O(f(n)\log^2 n )}$ in \cref{thm:sep} can be replaced by $2^{O(f(n)+\log^2 n )}$. The motivation for this question is the following: on the one hand, we have seen that for planar graphs we can improve the bound of \cref{thm:sep} from $2^{O(\sqrt{n}\log^2 n)}$ to $2^{O(\sqrt{n})}$; on the other hand, it is known that for $n$-vertex trees (which admit balanced separators of size 1), the minimum size of the labels in a distance labelling scheme is  $(\tfrac14+o(1)) \log^2 n$ \cite{FGNW17} and the constant $\tfrac14$ is best possible \cite{alstrup.gortz.ea:distance}. This shows in particular that the $\log^2 n$ term cannot be avoided in \cref{thm:sep} and in a possible improvement with $2^{O(f(n)+\log^2 n )}$ vertices. The work on distance labelling in trees mentioned above also motivates the following natural question: What is the smallest constant $c>0$ such that the class of $n$-vertex trees has an isometric-universal graph with at most $2^{c \log ^2 n}$ vertices? The lower bound on distance labelling schemes in trees \cite{alstrup.gortz.ea:distance} shows that $c\ge \tfrac14$, while known upper bounds on the size of trees containing all $n$-vertex trees as subgraphs~\cite{GL68,CGC81} (and thus also as isometric subgraphs) show that $c\le \tfrac12+o(1)$.

\section*{acknowledgment}

We thank Peter Winkler for asking \cref{q:winkler}, and for all the
subsequent discussions.

\bibliographystyle{abbrv}
\bibliography{bibliography}
\end{document}